\documentclass[12pt, reqno]{amsart}
\usepackage[utf8]{inputenc}

\usepackage{amsfonts}
\usepackage{amsmath}
\usepackage{amssymb}
\usepackage{amsthm}
\usepackage{enumerate}
\usepackage{comment}
\usepackage{esint}
\usepackage{xcolor}
\usepackage{geometry}
\usepackage{soul}
\newtheorem{theorem}{Theorem}

\newtheorem{lemma}[theorem]{Lemma}
\newtheorem{proposition}[theorem]{Proposition}
\newtheorem{definition}{Definition}

\newtheorem{remark}{Remark}

\newcommand{\R}{\mathbb{R}}

\newcommand{\Jop}{\overline{J}^{2,+}}

\def\Om{\Omega}

\def\vep{\varepsilon}
\def\ol{\overline}

\def\R{{\mathbb R}}
\def\X{{\mathbf X}}
\def\Y{{\mathbf Y}}

\def\L{{\mathcal L}}

\def\S{{\mathcal S}}

\def\O{{\mathcal O}}

\def\U{{\mathcal U}}
\def\V{{\mathcal V}}
\def\W{{\mathcal W}}

\DeclareMathOperator{\tr}{tr}

\title[Superposition in disjoint variables]{Superposition property in disjoint variables for the infinity Laplace equation}
\author[Q. Liu]{Qing Liu}
\address[Qing Liu]{Geometric Partial Differential Equations Unit, Okinawa Institute of Science and Technology Graduate University, Okinawa 904-0495, Japan, {\tt qing.liu@oist.jp}}

\author[J. J. Manfredi]{Juan J. Manfredi}
\address[Juan J. Manfredi]{Department of Mathematics, University of Pittsburgh, 312 Thackeray Hall, Pittsburgh, PA 15260, USA. {\tt manfredi@pitt.edu}}

\author[X. Zhou]{Xiaodan Zhou}
\address[Xiaodan Zhou]{Analysis on Metric Spaces Unit, Okinawa Institute of Science and Technology Graduate University, Okinawa 904-0495, Japan, {\tt xiaodan.zhou@oist.jp}}

\date{\today}

\begin{document}

\begin{abstract}
We establish a superposition principle in disjoint variables for the inhomogeneous infinity Laplace equation. We show that the sum of viscosity solutions of the inhomogeneous infinity Laplace equation in separate domains is a viscosity solution in the product domain. This result has been used in the literature with certain particular choices of solutions to simplify regularity analysis for a general inhomogeneous infinity Laplace equation by reducing it to the case without sign-changing inhomogeneous terms and vanishing gradient singularities. We present a proof of this superposition principle for general viscosity solutions. We also explore generalization in metric spaces using cone comparison techniques and study related properties for general elliptic and convex equations. 
\end{abstract}


\maketitle

\section{Introduction}
It is well known that the sum of two harmonic functions is harmonic. We say that the solutions to an equation have the superposition property when the sum of two solutions remains a solution to the equation. This property holds for linear equations, but in general fails for the nonlinear equations. For example, consider the infinity Laplace equation 
\begin{equation}\label{inftyL}
    -\Delta_\infty u=-\langle\nabla^2 u\nabla u, \nabla u\rangle=0
\end{equation}
in the domain $\Omega=\{(x_1, x_2)\in \mathbb{R}^2: x_1>0, x_2>0\}$. We can easily verify that the sum of two smooth infinity harmonic functions $u(x_1, x_2)=x_1^{\frac{4}{3}}-x_2^{\frac{4}{3}}$ and $v(x_1, x_2)=-x_1$ is not a solution to the infinity Laplace equation. However, we observe that the disjoint sum 
\[
w(x_1, x_2, x_3, x_4)=u(x_1, x_2)+v(x_3, x_4) =x_1^{\frac{4}{3}}-x_2^{\frac{4}{3}}-x_3
\]
is a solution of \eqref{inftyL} in $\Omega\times \Omega\subset \mathbb{R}^4$. This observation can be verified for viscosity solutions of the general inhomogeneous infinity Laplace equation.
\begin{equation}\label{inf-laplace}
    -\Delta_\infty u=f \quad \text{in $\Omega$,}
\end{equation}
where $\Omega$ is a domain in Euclidean space and $f\in C(\Omega)$ is a given function. When $u_i$ are viscosity solutions to \eqref{inf-laplace}, respectively, for $f=f_i$ and $\Omega=\Omega_i$ with $i=1, 2$, we will show that $u(x, y)=u_1(x)+u_2(y)$ is a viscosity solution to \eqref{inf-laplace} with $\Omega=\Omega_1\times \Omega_2$ and $f(x,y)=f_1(x)+f_2(y)
$. See Theorem~\ref{infsuper} for a precise statement of this superposition result.

Such a superposition in disjoint variables plays an important role in understanding properties of the infinity Laplace equation. A notable application appears in the paper \cite{Lin14} by Lindgren, 
where the regularity of solutions to \eqref{inf-laplace} is studied for a bounded domain $\Omega\subset \R^d$ and a bounded, possibly sign changing inhomogeneous term $f\in C(\Omega)$. For a given viscosity solution $u$,  the superposition property enables us to show that 
\[
w(x, y_1, y_2)=u(x)+a y_1+by_2^{\frac{4}{3}}, \quad x\in \Omega,\ y_1, y_2\in \R,
\]
with $a, b>0$ appropriately chosen, is a viscosity solution of 
\[
-\Delta_\infty w = g \quad \text{in $\Omega':=\Omega\times \R^2$,}
\]
\par\noindent  where the new inhomogeneous term $g$ satisfies $\inf_{\Omega'} g>0$. Moreover, the new equation is easier to handle, as the construction of $w$ also generates a positive lower bound for $|\nabla w|$. 
This reduction significantly simplifies the regularity problem by transforming the general equation into the more manageable case with a positive inhomogeneous term and no singularities at vanishing gradients. Similar techniques are applied to study the regularity of solutions to the inhomogeneous infinity Laplace equation associated with a frame of vector fields \cite{FM}. 


Although the superposition property in disjoint variables for
the infinity Laplacian has been mentioned in \cite{LindBook} and applied in \cite{Lin14} for certain specific solutions; a rigorous proof has not been provided yet for general viscosity solutions. Hong and Feng \cite{HF} give a proof of this property in the case when either of the solutions in the sum is smooth. 

The main purpose of this paper is therefore to present a rigorous proof for the superposition of two viscosity solutions of the infinity Laplace equation.  Our proof is based on the Theorem on Sums \cite[Theorem 3.2]{CIL}. We will also use our approach to study related problems for the Laplace equation and other second-order elliptic equations. 
It is possible to further consider a general elliptic operator in the form of 
\[
\L[u]=L(x, \nabla u, \nabla^2 u), 
\]
where $L: \R^d \times \R^d\times \S^d\to \R$ is a continuous function satisfying
\begin{equation}\label{ellip}
L(x,  \xi, X)\leq L(x, \xi, Y), \quad \text{for all $x\in \mathbb{R}^d$, $\xi\in \mathbb{R}^d, X, Y\in \S^d$ with $X\geq Y$. }
\end{equation}
Here, $\S^d$ denotes the set of all $d\times d$ symmetric matrices. Such an elliptic general operator is not the main focus of this paper. We refer the reader to our forthcoming paper \cite{LMZ24} for discussion on this topic for general quasilinear elliptic equations and related applications. 

In Section \ref{convex}, we adapt our arguments to study the standard superposition property (in common variables instead of disjoint ones) for viscosity subsolutions of certain nonlinear elliptic equations. Our result only applies to viscosity subsolutions under the sub-additivity assumption on the elliptic operator $L$. 

In \cite[Theorem 5.8]{CaCr}, Caffarelli and Cabr'e prove a superposition-like result for viscosity subsolutions to convex elliptic equations. 
\[
L(\nabla^2 u)=0.
\]
Under the convexity of $L$, they show that $(u+v)/2$ is a viscosity subsolution if both $u$ and $v$ are viscosity subsolutions. This result is also related to ours. We use our approach based on the Theorem on Sums to provide a proof of their result in the setting of more general convex elliptic equations 
\[
L(\nabla u, \nabla^2 u)=0. 
\]

Section \ref{metric} is an attempt to extend the superposition principle in disjoint variables to infinity harmonic functions in general metric spaces. 
It is well known that infinite harmonic functions in the Euclidean space can be characterized by comparison with cones \cite{ArCrJu}, and this equivalent definition can be generalized to metric spaces \cite{Ju, JuSh}. It turns out that by using the cone comparison characterization, one can generalize the superposition principle in disjoint variables in a product metric space if the $\ell^1$ metric for the product space is chosen. However, it is not clear whether comparison with cones is sufficient to yield the same result for other product metrics including the $\ell^2$ metric as in the Euclidean space. 
 
\subsection*{Acknowledgments}
We express our thanks to the anonymous referee for a careful reading and thoughtful suggestions that have improved the readability of this manuscript.

The work of QL was supported by the JSPS Grant-in-Aid for Scientific Research (No.~22K03396). The work of JM was supported by the Simmons Collaboration Award 962828. The work of XZ was supported by the JSPS Grant-in-Aid for Early-Career Scientists (No.~22K13947) and by the Research Institute for Mathematical Sciences, an International Joint Use/ Research Center located at Kyoto University. 


\section{Preliminaries}
In this section, we review some preliminaries, including the definition of viscosity solutions and the statement of the Theorem on Sums.

Let $\Omega\subset \R^d$ be a domain.
Let $L: \R^d \times \R^d\times \S^d\to \R$ be a continuous function satisfying \eqref{ellip}. 
In the following, we give the definition of viscosity solutions of the equation 
\begin{equation}\label{general eq}
    L(x, \nabla u, \nabla^2 u)=0 \quad \text{in $\Omega$}. 
\end{equation}
We denote by $USC(\Omega)$ the class of upper semi-continuous functions in $\Omega$ and by $LSC(\Omega)$ the class of lower semi-continuous functions in $\Omega$.
\begin{definition}[Viscosity solutions]\label{def-vis}
A function $u\in USC(\Om)$ is called a \emph{viscosity subsolution} of \eqref{general eq} if whenever there exist $x_0\in \Om$ and $\psi\in C^2(\Om)$ such that $u-\psi$ attains a local maximum (resp., local minimum) at $x_0$, we have 
\[
L(x_0, \nabla \psi(x_0), \nabla^2 \psi(x_0))\leq 0.
\]
A function $u\in LSC(\Om)$ is called a \emph{viscosity supersolution} of \eqref{general eq} if whenever there exist $x_0\in \Om$ and $\psi\in C^2(\Om)$ such that $u-\psi$ attains a local minimum at $x_0$, we have 
\[
L(x_0, \nabla \psi(x_0), \nabla^2 \psi(x_0))\geq 0.
\]
A function $u\in C(\Om)$ 
is called a \emph{viscosity solution} of \eqref{general eq} if $u$ is both a viscosity subsolution and a viscosity supersolution of \eqref{general eq}. 
\end{definition}

For $f\in C(\Omega)$, we obtain the definition of viscosity solutions to the inhomogeneous infinity Laplace equation \eqref{inf-laplace} taking in definition \ref{def-vis}
\[
L(x, \xi, X)=-\left\langle X \xi, \xi\right\rangle-f(x), \quad x\in \Omega,\  \xi \in \R^d,\ X\in \S^d. 
\]

As a standard remark in viscosity solution theory, one can alternatively use the so-called semijets to define viscosity solutions. Recall that for $u\in USC(\Om)$ and $x\in \Om$, we define 
\[
\begin{aligned}
J^{2, +}u(x)&=\bigg\{(\xi, X)\in \R^{d}\times \S^d: 
u(y)\le u(x) +\langle \xi, y-x\rangle\\
&+ \frac{1}{2}\langle X(y-x), y-x\rangle+o(|y-x|^2)\text{ as $y\in \Om$ and $y\to x$}\bigg\}.
\end{aligned}
\]
The semijet $J^{2, -}u(x)$ for $u\in LSC(\Om)$ is defined as $-J^{2, +}(-u)(x)$. 

It is known that, when $L: \Om\times \R^d\times \S^d\to \R$ is continuous, a locally bounded $u\in USC(\Om)$ (resp., $u\in LSC(\Om)$) is a viscosity subsolution (resp., supersolution) of \eqref{general eq} if and only if at any $x_0\in \Om$,
\[
\begin{aligned}
&L(x_0, \xi, X)\leq 0\quad \text{for all $(\xi, X)\in J^{2, +}u(x_0)$}\\
&\big(\text{resp.,} \ L(x_0, \xi, X)\geq 0\quad \text{for all $(\xi, X)\in J^{2, -}u(x_0)$}
\big);
\end{aligned}
\]
see, for example, \cite[Remark 2.3, Remark 8.1]{CIL} for discussions on this equivalence. 

Moreover, under the continuity assumption of $L$, the semijets $J^{2, \pm}u(x_0)$ in the equivalent definition above can be replaced by their closures $\overline{J}^{2, \pm}u(x_0)$, which are defined as
\[
\begin{aligned}
    \overline{J}^{2, \pm}u(x)= \bigg\{(\xi, X)\in & \ \R^{n}\times \S^n: \exists (x_j, \xi_j, X_j)\in \Om\times  \R^{n}\times \S^n\\
    & \text{ such that } (\xi_j, X_j)\in J^{2, \pm} u(x_j) \text{ and }\\
    & (x_j, u(x_j), \xi_j, X_j)\to (x, u(x), \xi, X)\text{ as $j\to \infty$}\bigg\}. 
\end{aligned}
\]
We refer to \cite{CIL} for more details concerning the definition of viscosity solutions involving semijets. 
    
The main ingredient in proving the superposition property is the Theorem on sums (or Crandall-Ishii-Jensen lemma) given below \cite[Theorem 3.2]{CIL}. 

\begin{lemma}[Theorem on sums]\label{TOS}
Let $\U\subset \R^n$ and $\V\subset \R^m$ be two domains. Let $u:\U\to \R$ and $v:\V \to \R$ be upper semicontinuous, and $\varphi\in C^2(\U\times \V)$. Assume that \[ (x, y)\mapsto u(x)+v(y)-\varphi(x,y) \] reaches a maximum at $(x_0,y_0)\in \U\times \V$. For any $\varepsilon>0$, there exist $(\xi, X)\in \overline{J}^{2,+}u(x_0)$ and $(\eta, Y)\in \overline{J}^{2,+}v(y_0)$ such that
\[
\xi=\nabla_x \varphi(x_0,y_0),\quad\eta=\nabla_y \varphi(x_0,y_0),    
\]
and
\[
-\left(\frac{1}{\varepsilon}+\|Z\|\right)\begin{pmatrix}
I_n & 0 \\
0 & I_m 
\end{pmatrix}\le \begin{pmatrix}
X & 0 \\
0 & Y
\end{pmatrix}\le Z+\varepsilon Z^2,
\]
where $Z=\nabla^2 \varphi(x_0, y_0)\in \S^{n+m}$. 
\end{lemma}

We say that the superposition principle in disjoint variables holds for a family of operators $L_n$, $n\in\mathbb{N}$ if for two domains $\U\subset \R^{n}$, $\V\subset \R^m$ and $f\in C(\U)$, $g\in C(\V)$, when $u\in C(\U)$, $v\in C(\V)$ are, respectively, a viscosity solution of the 
\[
L_n(x, \nabla u, \nabla^2 u)=f \quad \text{in $\U$}
\]
(with dimension $d=n$ for $L$) and a viscosity solution of 
\[
L_m(y, \nabla v, \nabla^2 v)=g \quad \text{in $\V$}
\]
(with dimension $d=m$ for $L$), the disjoint sum $w\in C(\U\times \V)$ given by  
\[
w(x, y)=u(x)+v(y), \quad (x, y)\in \U\times \V\subset \R^m\times \R^n
\]
is a viscosity solution of 
\[
L_d(x, y, \nabla w(x, y), \nabla^2 w(x, y))=f(x)+g(y) \ \quad \text{for $(x, y)\in \U\times \V$}
\]
(with dimension $d=n+m$ for $L$). 
\section{Main results}
\subsection{Superposition in disjoint variables}\label{main}
The following is our main result: 

\begin{theorem}[Superposition in disjoint variables for the infinity Laplacian]\label{infsuper}
 Let $\U\subset \mathbb{R}^n$ and $\V\subset \R^m$ be domains. Suppose $f\in C(\U)$ and $g\in C(V)$. Assume that $u$ is a viscosity subsolution (resp., supersolution, solution) to 
   \begin{equation}\label{u-inf}
   -\Delta_\infty u(x)=f(x)\quad \text{in $\U$}
   \end{equation}
    and $v$ is a viscosity subsolution (resp., supersolution, solution)  to 
   \begin{equation}\label{v-inf}
   -\Delta_\infty v(y)=g(y)\quad \text{in $\V$.}
   \end{equation}
Then $w(x,y)=u(x)+v(y)$ is a viscosity subsolution (resp., supersolution, solution) to 
   \begin{equation}\label{w-inf}
   -\Delta_\infty w(x, y)=f(x)+g(y) \quad \text{in $\U\times \V\subset \mathbb{R}^n\times \mathbb{R}^m$.}
   \end{equation}    

\end{theorem}

\begin{proof}
    Let us prove the result for subsolutions $u\in USC(\U)$ and $v\in USC(\V)$. Assume that $\varphi\in C^2(\U\times \V)$ and $u(x)+v(y)-\varphi(x,y)$ attains a local maximum at $(x_0,y_0)\in \U\times \V$. Let $\varepsilon>0$. For every $\varepsilon>0$, by Lemma \ref{TOS} (Theorem on sums), there exist $(\xi, X_\varepsilon)\in \overline{J}^{2,+}u(x_0)$ and $(\eta, Y_\varepsilon)\in \overline{J}^{2,+}v(y_0)$ such that \[
    \xi=\nabla_x \varphi(x_0,y_0),\quad\eta=\nabla_y \varphi(x_0,y_0),
    \]and 
    \begin{equation}\label{2ndest}  
 \begin{pmatrix}
X_\varepsilon & 0 \\
0 & Y_\varepsilon
\end{pmatrix}\le Z+\varepsilon Z^2,
\end{equation}
where $Z=\nabla^2 \varphi(x_0, y_0)\in \S^{n+m}$.
  Multiplying \eqref{2ndest} by the row vector $\nabla \varphi(x_0,y_0)$ from left and by its transpose from right, we obtain
    \[
    \begin{aligned}
    \langle X_\varepsilon \xi, \xi\rangle+\langle Y_\varepsilon \eta, \eta\rangle \le 
       \langle \nabla^2 \varphi(x_0,y_0)\nabla \varphi(x_0,y_0), \nabla \varphi(x_0,y_0)\rangle+C\varepsilon 
    \end{aligned}
    \]
     for some constant $C>0.$ 
     
     Since $u$ is a viscosity subsolution to $-\Delta_\infty u=f$ and $(\xi, X_\varepsilon)\in \Jop u(x_0)$, we have
     \[
    - \langle X_\varepsilon \xi, \xi\rangle\le f(x_0). 
     \]
Likewise, by the supersolution property of $v$ we get
\[
    - \langle Y_\varepsilon \eta, \eta\rangle\le g(y_0).
     \]
Combining the above inequalities, we have 
\[
-\langle \nabla^2 \varphi(x_0,y_0)\nabla \varphi(x_0,y_0), \nabla \varphi(x_0,y_0)\rangle\le f(x_0)+g(y_0)+C\varepsilon.
\]
Letting $\varepsilon\to 0$, we complete the proof. 

The above argument also applies when $u,v$ are viscosity supersolutions or solutions instead of subsolutions. Hence the superposition property in disjoint variables holds for the non-homogeneous infinity Laplace equation.
    \end{proof}

 %

\begin{remark} 
 For $\Om\subset \R$, $K\in \R$, Hong and Feng \cite{HF} consider a point-wise superposition of viscosity solutions to the following inhomogeneous infinity Laplace equation
\begin{equation}\label{infK}
-\Delta_\infty u=K \text{ in $\Om$}.
\end{equation}
In \cite[Definition 1]{HF}, a function $u\in USC(\Om)$ is called a subsolution of $-\Delta_\infty u(x_0)\le K$ at a point $x_0\in \R$ in the viscosity sense if \[ -\Delta_\infty \phi (x_0)\le K \] is true whenever there exists a test function $\phi\in C^2(\Om)$ such that $u-\phi$ attains a local maximum at $x_0$. Note that this notion is restricted to tests at $x_0$ and therefore is much weaker than the usual definition 
of viscosity subsolutions to \eqref{infK} in the whole domain $\Om$. 

In a domain $\Om$ containing $0$, Hong and Feng constructed in \cite[Section 3]{HF} two functions $u$ and $v$ satisfying
\[
-\Delta_\infty u(0)\le -1 \quad \text{ and } -\Delta_\infty v(0)\le -1,
\]
but 
\[
-\Delta_\infty (u+v)(0) > -2 
\]
under their notion of viscosity subsolutions and supersolutions at $0$. This example does not contradict our main theorem, since $u,v$ are not necessarily viscosity solutions to \eqref{infK} in $\Om$ in the usual sense. In fact, our result of Theorem \ref{infsuper} is obtained by using Theorem on sums, whose proof involves appropriate perturbations and requires the subsolution property of $u$ a neighborhood of $x_0$ and of $v$ in a neighborhood of $y_0$. The subsolution property only at the points $x_0$ or $y_0$ is not sufficient to yield our superposition result. 
\end{remark}

\begin{remark}\label{Lsuperdis}
    The argument for Theorem \ref{infsuper} can be easily adapted to obtain the superposition property in disjoint variables for inhomogeneous Laplace equations. Following the same proof of Theorem \ref{infsuper}  and taking trace on both sides of \eqref{2ndest}, we obtain that
    \[
\tr(X_\varepsilon)+\tr(Y_\varepsilon)\le \tr(\nabla^2\varphi(x_0,y_0))+C\varepsilon.
    \]
When $-\Delta u\leq f(x)$ in $\U$ and $-\Delta v\leq g(y)$ in $\U$ hold in the viscosity sense, we have
\[
-\Delta\varphi(x_0,y_0)\le f(x_0)+g(y_0)+C\varepsilon,
    \]
which implies 
 \[
-\Delta\varphi(x_0,y_0)\le f(x_0)+g(y_0)
    \]
    by sending $\varepsilon$ to zero. We can similarly prove the supersolution property as well. 
\end{remark}

\begin{remark}
    Combining Theorem \ref{infsuper} and Remark \ref{Lsuperdis}, one can deduce the property of superposition in disjoint variables for an inhomogeneous equation with $\mathcal{L}[u]=-\Delta u-\Delta_\infty u$. In fact, the infinity Laplace equation and Laplace equation are prototypes of two general classes of elliptic equations for which the superposition in disjoint variables holds. Likewise, this property can be extended to linear combinations of these two classes of equations. More details and examples will be given in \cite{LMZ24}. 
\end{remark}

\subsection{Superposition in common variables for sub-additive elliptic equations} \label{convex} 

It is easily seen that the superposition principle in common variables (rather than disjoint ones) holds for smooth harmonic functions. In the following, we use our arguments to give a proof for viscosity solutions.


\begin{proposition}[Superposition for Laplacian]\label{Lapsup}
    Let $\Omega\subset \mathbb{R}^n$ be a domain. Assume that $u, v$  are viscosity subsolutions (resp., supersolutions, solutions) to 
    \[
    -\Delta u=0 \quad \text{in $\Omega$}.
    \] Then $u+v$ is a viscosity subsolution (resp., supersolution, solution) to the same equation. 
\end{proposition}

\begin{proof}
Let $\varphi\in C^2(\Omega)$ and $\O\subset\subset \Omega$ be an open and bounded set. 
Assume that $u(x)+v(x)-\varphi(x)$ reaches a maximum in $\overline{\O}$ at $x=x_0$. Without changing the value of $-\Delta \varphi(x_0)$, we can add $|x-x_0|^4$ to $\varphi$ and assume that $x_0$ is the unique maximizer of $u+v-\varphi$ in $\ol{\O}$.   Suppose that
\[
\max_{x\in \overline{\O}}(u(x)+v(x)-\varphi(x))=u(x_0)+v(x_0)-\varphi(x_0)=\delta.
\]
Let $\varepsilon>0$. Consider $\Phi_\varepsilon:\overline{\O}\times\overline{\O}\to \R$ defined as
\[
\Phi_\varepsilon(x,y)=u(x)-\varphi(x)+v(y)-\frac{|x-y|^2}{2\varepsilon}.
\]
Then there exists $(x_\varepsilon, y_\varepsilon)\in \overline{\O}\times \overline{\O}$ such that
\[
\max_{\overline{\O}\times\overline{\O}}\Phi_\varepsilon(x,y)=\Phi(x_\varepsilon, y_\varepsilon)\ge \delta,
\]
which implies
\[
\frac{|x_\varepsilon-y_\varepsilon|^2}{2\varepsilon}\le u(x_\varepsilon)-\varphi(x_\varepsilon)+v(y_\varepsilon)-\delta.
\]
Due to the boundedness of $u$ and $v$ above, we therefore have $|x_\vep-y_\vep|\to 0$ as $\vep\to 0$. Since $\overline{\O}$ is compact, there exist subsequences such that $x_\varepsilon\to \hat{x}$ and $y_\varepsilon\to \hat{x}$ for some $\hat{x}\in \overline{\O}$. We still index the subsequences by $\vep$ for simplicity. 

By the upper semicontinuity of $u$ and $v$, we let $\vep\to 0$ to get
\[
\limsup_{\varepsilon\to 0}\frac{|x_\varepsilon-y_\varepsilon|^2}{2\varepsilon}\le u(\hat{x})-\varphi(\hat{x})+v(\hat{x})-\delta,
\]
which implies $u(\hat{x})-\varphi(\hat{x})+v(\hat{x})\geq \delta$. 
Since $u+v-\varphi$ achieves a unique maximum in $\ol{\O}$ at $x_0$, we therefore have $x_0=\hat{x}$. 


Applying Lemma \ref{TOS} to 
\[
(x, y)\mapsto (u(x)-\varphi(x))+v(y)-\frac{|x-y|^2}{2\varepsilon},
\] 
we obtain 
\[
(\xi_\varepsilon, X_\varepsilon)\in \Jop(u-\varphi)(x_\varepsilon)\quad\text{ and } (\eta_\varepsilon, Y_\varepsilon)\in \Jop(v)(y_\varepsilon)\]
with
\[
\xi_\varepsilon=\frac{x_\varepsilon-y_\varepsilon}{\varepsilon}, \quad \eta_\varepsilon=\frac{y_\varepsilon-x_\varepsilon}{\varepsilon}
\]
and
\[
\begin{pmatrix}
X_\varepsilon & 0 \\
0 & Y_\varepsilon 
\end{pmatrix}\le \frac{3}{\varepsilon}\begin{pmatrix}
I & -I \\
-I & I 
\end{pmatrix}.
\]
Let $\zeta\in \mathbb{R}^n$ be an arbitrary vector. Multiplying the above inequality by $(\zeta, \zeta)$ from the left and its transpose $(\zeta, \zeta)^T$ from the right gives 
\begin{equation}\label{negH}
\langle(X_\varepsilon+Y_\varepsilon)\zeta, \zeta\rangle \le 0. 
\end{equation}
Hence, $\tr(X_\varepsilon)+\tr(Y_\varepsilon)\le 0.$
On the other hand, since $\varphi\in C^2(\Omega)$, we can easily verify that 
\[
(\xi_\varepsilon+\nabla \varphi(x_\varepsilon), \ X_\varepsilon+\nabla^2\varphi(x_\varepsilon))\in \Jop u(x_\varepsilon).
\]
By the fact that $u,v$ are viscosity subsolutions to $-\Delta u=0$, we have
\[
-\tr(X_\varepsilon+\nabla^2\varphi(x_\varepsilon))\le 0
\]
and 
\[
-\tr(Y_\varepsilon)\le 0.
\]
Adding the above inequalities and applying $\tr(X_\varepsilon)+\tr(Y_\varepsilon)\le 0$, we get
\[
-\Delta \varphi(x_\vep)=-\tr(\nabla^2\varphi(x_\varepsilon))\le 0
\]
 We conclude $-\Delta \varphi(x_0)\leq 0$ by letting $\varepsilon\to 0$. The proof is thus complete. 
\end{proof}

The above proposition is a special case of the superposition property for more general operators $L(\xi, X)$ satisfying the following subadditivity.
\begin{equation}\label{sub-add}
L(\xi+\eta, X+Y)\leq L(\xi, X)+L(\eta, Y) \quad \text{for $\xi, \eta\in \R^n$, $X, Y\in \S^n$.}
\end{equation}

\begin{theorem}[Subsolutions of sub-additive equations]\label{super-common}
   Let $\Omega\subset \mathbb{R}^n$ be a domain, and $f, g\in C(\Omega)$. Assume that $L:\R^n\times \S^n\to \R$ is a continuous elliptic operator satisfying \eqref{sub-add}. 
Assume that $u\in USC(\Omega)$ is a viscosity subsolution of \begin{equation}\label{conL}
     L(\nabla u,\nabla^2 u)=f\quad \text{in $\Omega$}
     \end{equation}
     and $v\in USC(\Omega)$ is a viscosity subsolution of 
     \begin{equation}\label{conL2}
     L(\nabla v,\nabla^2 v)=g\quad \text{in $\Omega$}.
     \end{equation}
    Then, $w=u+v$ is a viscosity subsolution to \begin{equation}\label{conL3}
     L(\nabla w,\nabla^2 w)=f+g\quad \text{in $\Omega$}.
     \end{equation} 
\end{theorem}

\begin{proof}
     Assume that $u(x)+v(x)-\varphi(x)$ attains a strict maximum in $\overline{\O}$ in $x_0\in \O$ for some $\varphi\in C^2(\Omega)$ and ${\O}\subset\subset \Omega$. The same argument as in the proof of Proposition \ref{Lapsup} implies the existence of a sequence $x_\vep\to x_0$ and 
\[
(\xi_\varepsilon+\nabla \varphi(x_\varepsilon), X_\varepsilon+\nabla^2 \varphi(x_\varepsilon))\in \Jop u(x_\varepsilon), \quad(\eta_\varepsilon, Y_\varepsilon)\in \Jop v(y_\varepsilon),
\]
with $\xi_\varepsilon=-\eta_\varepsilon$, and $X_\varepsilon+Y_\varepsilon\le 0$. 

It follows from the ellipticity of $L$ and the condition \eqref{sub-add} that
\begin{equation}\label{test-L}
\begin{aligned}
L(\nabla \varphi(x_\varepsilon), \nabla^2\varphi(x_\varepsilon))&\le L\left(\nabla \varphi(x_\varepsilon),\ \nabla^2\varphi(x_\varepsilon)+ X_\varepsilon+Y_\varepsilon\right)\\
&=L\left(\xi_\varepsilon+\nabla \varphi(x_\varepsilon)+\eta_\vep, \nabla^2\varphi(x_\varepsilon)+ X_\varepsilon+Y_\varepsilon\right)\\
&\le L\left(\xi_\varepsilon+\nabla \varphi(x_\varepsilon), X_\varepsilon+\nabla^2 \varphi(x_\varepsilon)\right)+L(\eta_\varepsilon, Y_\varepsilon).
\end{aligned}
\end{equation}
Then the subsolution property of $u$ and $v$ immediately yields 
\[
L(\nabla \varphi(x_\varepsilon), \nabla^2\varphi(x_\varepsilon))\leq f(x_\vep)+g(y_\vep).
\]
Passing to the limit as $\vep\to 0$, we obtain 
\[
L(\nabla \varphi(x_0), \nabla^2\varphi(x_0))\leq f(x_0)+g(x_0),
\]
as desired. 
\end{proof}

We can obtain the same property for viscosity supersolutions if we assume the superadditivity of $L$: 
\[
L(\xi+\eta, X+Y)\geq L(\xi, X)+L(\eta, Y) \quad \text{for $\xi, \eta\in \R^n$, $X, Y\in \S^n$.}
\]
The superposition principle for viscosity solutions thus holds when $L$ is linear.


Another related result is due to Caffarelli and Cabr\'e \cite[Theorem 5.8]{CaCr}, who showed that, for a convex elliptic operator $L:\S^n\to \R$,  if $u,v\in USC(\Omega)$ are viscosity subsolutions to 
\[
L(\nabla^2 u)=0\quad \text{in $\Omega$,}
\] 
then $(u+v)/2$ is also a viscosity subsolution to the same equation. 

We can slightly generalize \cite[Theorem 5.8]{CaCr} for a convex elliptic operator depending also on the gradient, adapting the arguments in the proof of Theorem \ref{super-common}. 

\begin{proposition}[Subsolutions of convex equations]\label{prop-convex}
    Let $\Omega\subset \mathbb{R}^n$ be a domain.  Assume that $L:\R^n\times \S^n\to \R$ is a continuous elliptic operator satisfying 
    \begin{equation}\label{convexL}
\frac{1}{2}L(\xi, X)+\frac{1}{2}L(\eta, Y)\ge L\left(\frac{\xi+\eta}{2}, \frac{X+Y}{2}\right).
\end{equation}
Assume that $u, v\in USC(\Omega)$ are viscosity subsolutions of 
\[
L(\nabla u, \nabla^2 u)=0 \quad \text{in $\Omega$.}
\]
Then, $\displaystyle \frac{u+v}{2}$ is a viscosity subsolution of the same equation. 
\end{proposition}
\begin{proof}
The proof resembles that of Theorem \ref{super-common}.  For any test function $\varphi\in C^2(\Omega)$ for $(u+v)/2$ at $x_0$, we have 
\[
(\xi_\varepsilon+2\nabla \varphi(x_\varepsilon), X_\varepsilon+2\nabla^2 \varphi(x_\varepsilon))\in \Jop u(x_\varepsilon), \quad(\eta_\varepsilon, Y_\varepsilon)\in \Jop v(y_\varepsilon),
\]
with $\xi_\varepsilon=-\eta_\varepsilon$, $X_\varepsilon+Y_\varepsilon\le 0$ for $x_\vep, y_\vep$ approaching $x_0$ as $\vep\to 0$.

By the ellipticity and \eqref{convexL}, instead of \eqref{test-L} we get
\[
\begin{aligned}
L(\nabla \varphi(x_\varepsilon), \nabla^2\varphi(x_\varepsilon))&\le L\left(\nabla \varphi(x_\varepsilon), \frac{2\nabla^2\varphi(x_\varepsilon)}{2}+\frac{X_\varepsilon+Y_\varepsilon}{2}\right)\\
&=L\left(\frac{2\nabla \varphi(x_\varepsilon)+\xi_\varepsilon}{2}+\frac{\eta_\varepsilon}{2}, \frac{2\nabla^2\varphi(x_\varepsilon)+X_\varepsilon}{2}+\frac{Y_\varepsilon}{2}\right)\\
&\le \frac{1}{2}L\left(2\nabla \varphi(x_\varepsilon)+\xi_\varepsilon,{2\nabla^2\varphi(x_\varepsilon)}+X_\varepsilon\right)+\frac{1}{2}L(\eta_\varepsilon, Y_\varepsilon).
\end{aligned}
\]
We obtain $L(\nabla\varphi(x_0), \nabla^2 \varphi(x_0))\leq 0$ by utilizing the definition of subsolutions for $u$ and $v$ and sending $\vep\to 0$.
\end{proof}

Note that the results in Theorem \ref{super-common} and Proposition \ref{prop-convex} do not apply to the infinite Laplacian. The operator $L$ for $-\Delta_\infty$ is given by 
\[
L(\xi, X)=-\left\langle X \xi, \xi\right\rangle, \quad \xi\in \R^n, \ X\in \S^n
\]
is neither subadditive nor convex in $\R^n\times \S^n$, although it is linear with respect to $X$ for any fixed $\xi$.

The following operators satisfy the convexity condition \eqref{convexL}. 
\begin{enumerate}
    \item[(a)] The viscous or inviscid Hamilton-Jacobi operator: 
    \[
    L(\xi, X)= -a \tr X+H(\xi), \quad \xi\in \R^n,\  X\in \S^n
    \]
    with $a\geq 0$ and $H$ convex in $\R^n$. 
    \item[(b)] The Pucci operator: 
    \[
    L(X)=- \min\{\tr(AX): \lambda I\leq  A\leq \Lambda I\}, \quad \Lambda\geq \lambda >0, \quad X\in \S^n. 
    \]
    \item[(c)] The inverse curvature operator for convex functions in one dimension: 
    \[
    L(\xi, X)= \frac{(1+\xi^2)^2}{X}, \quad \xi\in \R, \ X\in \R, \ X>0. 
    \]
    Note that for $\xi\in \R$, $X>0$ 
    \[
    \nabla^2 L(\xi, X) =\begin{pmatrix}
        2(1+\xi^2)^2 X^{-3} & -4\xi (1+\xi^2)X^{-2}  \\
        -4\xi (1+\xi^2)X^{-2} & (4+12\xi^2) X^{-1}
    \end{pmatrix}>0.
    \]
    (In fact, $\det(\nabla^2 L(\xi, X))=8(1+\xi^2)^3X^{-4}>0$). This operator is derived from the one dimensional parabolic equation
    \[
    \frac{u_t}{(1+u_x^2)^{\frac{1}{2}}}=-\left(\frac{u_{xx}}{(1+u_x^2)^{\frac{3}{2}}}\right)^{-1},
    \]
    where the left hand side is the normal velocity of a graph-like curve and the right hand side is the negative reciprocal of its curvature. 
\end{enumerate}
The Pucci operator is also subadditive. The Hamilton-Jacobi operator is subadditive if the function $H$ is.

\section{Superposition property in metric spaces}\label{metric}

It is well-known that comparison with cones can be employed to characterize the viscosity solutions of  
\begin{equation}\label{inf-lap}
-\Delta_\infty u=0 \quad \text{in $\Omega$}
\end{equation}
for a domain $\Omega$ in the Euclidean space; see, for example, \cite{ArCrJu}. Here we recall from \cite{Ju, JuSh} a generalization of this property in general metric spaces. 

\begin{definition}[Comparison with cones]\label{def cpc}
Let $(\X, d_\X)$ be a proper geodesic space, and $\Omega\subsetneq \X$ be a domain.
 A  function $u: \Omega\to \R$ that is locally bounded from above is 
said to 
satisfy comparison with cones from above in $\Omega$ (with respect to the metric $d_\X$) if the following property holds: for any $\hat{x}\in \Omega$, any $c\in \R$, $\kappa\geq 0$ and any bounded open set $\O\subset\subset \Omega$ with $\hat{x}\in\Omega\setminus\O$, the condition 
\begin{equation}\label{comp cone1}
u\leq \phi\quad \text{on $\partial \O$}
\end{equation}
for the cone function $\phi$ given by 
\begin{equation}\label{cone above}
\phi=c+\kappa d_\X(\hat{x}, \cdot)\quad \text{in $\Omega$}
\end{equation}
implies that 
\begin{equation}\label{comp cone2}
u\leq \phi\quad \text{in $\overline{\O}$.}
\end{equation}

We say that any $u: \Omega\to \R$ locally bounded from below satisfies comparison with cones from below in $\Omega$ (with respect to $d$) if $-u$ satisfies comparison from above in $\Omega$ as defined above. In addition, a locally bounded function $u: \Omega\to \R$ is said to satisfy comparison with cones (from both sides) in $\Omega$ (with respect to $d$) if it satisfies comparison with cones both from above and from below in $\Omega$. 
\end{definition}

\begin{theorem}[Superposition for comparison with cones]\label{thm superposition1}
Suppose that $(\X, d_\X)$ and $(\Y, d_{\Y})$ are proper geodesic spaces and that $\U\subset \X$ and $\V\subset \Y$ are two domains. Let $\U\times \V\subset \X\times \Y$ be the product space of $\U$ and $\V$ with metric $d_1$ given by 
\begin{equation}\label{product metric}
d_{1}((x_1, y_1), (x_2, y_2))= d_{\X}(x_1, x_2)+ d_{\Y}(y_1, y_2), \quad (x_1, y_1), (x_2, y_2)\in \X\times \Y.
\end{equation}
If $u\in C(\U)$, $v\in C(\V)$  satisfy comparison with cones from above respectively in $\U$ and in $\V$, then the disjoint sum $w\in C(\U\times \V)$ defined by 
\begin{equation}\label{disjoint sum}
w(x, y)=u(x)+v(y), \quad x\in \U,\  y\in \V    
\end{equation}
satisfies comparison with cones from above in $\U\times \V$ with respect to the metric $d_1$. Similar results for functions that satisfy comparison with cones from below and from both sides hold. 
\end{theorem}
\begin{proof}
Denote $\W=\U\times \V$. It is clear that $w\in C(\W)$ with respect to the metric $d_1$. Assume by contradiction that there exist a bounded open set $\O\subset\subset \W$, $\hat{z}=(\hat{x}, \hat{y})\in \W$ and a cone function 
\[
\phi(z)=c+\kappa d_{1}(z, \hat{z})=c+\kappa d_\X(x, \hat{x})+\kappa d_\Y(y, \hat{y}), \quad z=(x, y)\in \W
\]
with $c\in \R$, $\kappa\geq 0$ such that 
\begin{equation}\label{superposition1 eq1}
w\leq \phi\quad \text{on $\partial \O$}
\end{equation}
but $w-\phi$ takes a positive value somewhere in $\O$. Let $z_0=(x_0, y_0)$ (with $x_0\in \U$, $y_0\in \V$) be a maximizer of $w-\phi$ in $\O$. We thus have 
\begin{equation}\label{superposition1 eq2}
w(z_0)-\phi(z_0)>0.
\end{equation}
Take $\O_{\Y}=\{y\in \V: (\hat{x}, y)\in \O\}$ and let
\[
\phi_\Y(y)=\phi(\hat{x}, y)= c+\kappa d_\Y(y, \hat{y}), \quad y\in \V. 
\]
Note that $\phi_\Y$ is a cone function in $\V$. 
It is clear that $\hat{y}\notin \O_\Y$. Since $y\in \partial \O_\Y$ implies that $(\hat{x}, y)\in \partial \O$, by \eqref{superposition1 eq1} we have 
\[
w(\hat{x}, y)=u(\hat{x})+v(y)\leq \phi_\Y(y) \quad \text{for all $y\in \partial \O_\Y$}. 
\]
Applying the comparison with cones to $v$ in $\V$, we obtain $w(\hat{x}, \cdot)\leq \phi_\Y$ in $\O_\Y$. In particular, if $(\hat{x}, y_0)\in \O$, then 
\begin{equation}\label{metric revise1}
w(\hat{x}, y_0)\leq \phi_\Y(y_0)=\phi(\hat{x}, y_0).    
\end{equation}

Now define $\O_{\X}=\{x\in \U: (x,y_0)\in \O\}\setminus \{\hat{x}\}$. It is clear that $\O_\X$ is an open set in $\U$. Also, if $(\hat{x}, y_0)\in \O$, then $\hat{x}\in \partial \O_\X$; otherwise, $\partial \O_\X=\{x\in \U: (x, y_0)\in \partial \O\}$. In either case, by \eqref{superposition1 eq1} and \eqref{metric revise1}, we have
 $u\leq \phi_{\X}$ on $\partial \O_{\X}$, where
\[
\phi_{\X}(x)=c+\kappa d_\X(x, \hat{x})+\kappa d_\Y(y_0, \hat{y})-v(y_0), \quad x\in \U. 
\]
On the other hand, by \eqref{superposition1 eq2} we deduce that $u(x_0)>\phi_\X(x_0)$. Since $\phi_{\X}$ is a cone function in $\U$ and $\hat{x}\notin \O_{\X}$, we have a contradiction to the assumption that $u$ satisfies comparison with cones in $\U$. 
\end{proof}

Note that we here used the $\ell^1$ metric for $\W= \U\times \V$ in the product space $\X\times \Y$. 
It is not clear whether $w$ satisfies comparison with cones with the distance given by the $\ell^2$ metric. 
In Euclidean spaces and in Carnot groups, one can show the equivalence between the comparison with cones for metrics $d_1$ and $d_2$ using the definition of viscosity solutions. However, such a viscosity approach seems unavailable in general metric spaces.

\begin{thebibliography}{99}


\bibitem{ArCrJu}
G.~Aronsson, M.~G. Crandall, and P.~Juutinen,
\newblock A tour of the theory of absolutely minimizing functions.
\newblock {\em Bull. Amer. Math. Soc. (N.S.)}, 41(4):439--505, 2004.

\bibitem{CaCr} L.~Caffarelli, and X.~Cabr\'e, 
{\newblock \em Fully nonlinear elliptic equations},
Amer. Math. Soc. Colloq. Publ., 43
American Mathematical Society, Providence, RI, 1995, vi+104 pp.


\bibitem{CIL} M.~G. Crandall, H.~Ishii, and P.-L. Lions.
\newblock User's guide to viscosity solutions of second order partial
  differential equations.
\newblock {\em Bull. Amer. Math. Soc. (N.S.)}, 27(1):1--67, 1992.

\bibitem{FM} 
F.~Ferrari and J.~J. Manfredi.
\newblock A priori estimates for the {$\infty$}-{L}aplacian relative to vector
  fields.
\newblock {\em Funkcial. Ekvac.}, 66(1):45--70, 2023.



\bibitem{HF} G. Hong and X. Feng, 
\newblock Superposition principle on the viscosity solutions of infinity
  {L}aplace equations.
\newblock {\em Nonlinear Anal.}, 171:32--40, 2018.

\bibitem{Ju}
P.~Juutinen. 
\newblock Absolutely minimizing Lipschitz extensions on a metric space.
\newblock {\em Ann. Acad. Sci. Fenn. Math.}, 27(1):57--67, 2002.

\bibitem{JuSh}
P.~Juutinen and N.~Shanmugalingam.
\newblock Equivalence of {AMLE}, strong {AMLE}, and comparison with cones in
  metric measure spaces.
\newblock {\em Math. Nachr.}, 279(9-10):1083--1098, 2006.

\bibitem{LindBook}
P.~Lindqvist.
\newblock {\em Notes on the infinity {L}aplace equation}.
\newblock SpringerBriefs in Mathematics. BCAM Basque Center for Applied
  Mathematics, Bilbao; Springer, [Cham], 2016.


\bibitem{Lin14} 
E.~Lindgren.
\newblock On the regularity of solutions of the inhomogeneous infinity
  {L}aplace equation.
\newblock {\em Proc. Amer. Math. Soc.}, 142(1):277--288, 2014.
 



\bibitem{LMZ24} Q. Liu, J. Manfredi and X. Zhou, \newblock{Superposition principle in disjoint
variables for quasilinear elliptic equations  and applications}, \newblock{in preparation}.







\end{thebibliography}
\end{document}